\documentclass[pdflatex,sn-mathphys-num]{sn-jnl}

\usepackage{amsmath,amsfonts,amsthm,amssymb,graphics,color}

\usepackage{graphicx}%
\usepackage{multirow}%
\usepackage{amsmath,amssymb,amsfonts}%
\usepackage{amsthm}%
\usepackage{mathrsfs}%
\usepackage[title]{appendix}%
\usepackage{xcolor}%
\usepackage{textcomp}%
\usepackage{manyfoot}%
\usepackage{booktabs}%
\usepackage{algorithm}%
\usepackage{algorithmicx}%
\usepackage{algpseudocode}%
\usepackage{listings}%
\newcommand{\R}{\mathbb R}
\newcommand{\C}{\mathbb C}

\renewcommand{\Re}{\mathop {\rm Re}\nolimits}

\def\b#1{\boldsymbol{#1}}

\newtheorem{theorem}{Theorem}
\newtheorem{lemma}[theorem]{Lemma}
\newtheorem{remark}[theorem]{Remark}

\begin{document}

\title[Rational methods for abstract linear initial boundary value problems without order reduction]{Rational methods for abstract linear initial boundary value problems without order reduction}


\author[1]{\fnm{Carlos} \sur{Arranz-Sim\'on}}\email{carlos.arranz@uva.es}
\author*[1]{\fnm{Bego\~na} \sur{Cano}}\email{bcano@uva.es}
\author[1]{\fnm{C\'esar} \sur{Palencia}}\email{cesar.palencia@uva.es}
\affil*[1]{\orgdiv{IMUVA, Departamento de Matem\'atica Aplicada}, \orgname{Universidad de Valladolid}, \orgaddress{\street{Paseo de Bel\'en, 7}, \city{Valladolid}, \postcode{47011}, \state{Spain}, \country{Spain}}}
%
%%\affil[2]{\orgdiv{Department}, \orgname{Organization}, \orgaddress{\street{Street}, \city{City}, \postcode{10587}, \state{State}, %\country{Country}}}
%
%%\affil[3]{\orgdiv{Department}, \orgname{Organization}, \orgaddress{\street{Street}, \city{City}, \postcode{610101}, \state{State}, %\country{Country}}}
%
%%%==================================%%
%%% Sample for unstructured abstract %%
%%%==================================%%

\abstract{Given an $A$-stable rational approximation to $e^z$ of order $p$, numerical procedures are suggested to time integrate abstract, well-posed IBVPs, with time-dependent source term $f$ and boundary value $g$. These procedures exhibit the optimal order $p$ and can be implemented by using just one single evaluation of $f$ and $g$ per step, i.e., no evaluations of the derivatives of data are needed, and are of practical use at least for $p\le 6$. The full discretization is also studied and the theoretical results are corroborated by numerical experiments.}




\maketitle

\section{Introduction}\label{sec1}

It is well-known the phenomenom of order reduction which turns up when integrating evolutionary problems in partial differential equations with Runge--Kutta methods \cite{OR, OR2,SSVH, V}. Because of this, several techniques in the literature have been devised to avoid it.

Some of them are based on considering additional restrictions on the coefficients of the methods so that they not only satisfy the classical order conditions, but also some stiff order ones \cite{LV, LO} or, more recently, weak stage order conditions \cite{BKSS, BKSS2}. That implies less
freedom in the choice of coefficients, so that the error constants of the methods cannot be minimized in the same way, and also the number of stages which is required to obtain a certain order of accuracy may increase, and thus the computational cost of the method.

Another technique for linear problems was suggested in \cite{CP}, which consists of converting the problem, through the solution of several elliptic problems, to one for which order reduction is not observed. This procedure has the advantage to be valid for any method, but the solution of the corresponding elliptic problems also means a non-negligible computational cost.

On the other hand, another procedure is based on modifying the boundary values for the stages which are in some way preassumed \cite{AGC, A, A2, AC, AC2, AC3,CGAD, Pat}. That means very little computational cost because the number of nodes on the boundary is negligible with respect to the number of nodes on the whole domain. For linear problems, the expressions for the modified boundaries depend on spatial and time derivatives of data \cite{A2} (more particularly, the boundary condition and the source term). In many practical problems, such analytic expressions are not available, but just the values at some instants of time. Due to that, numerical differentiation is required to approximate the required modified boundary values for the stages, and it is well known that numerical differentiation is unstable when the grid is refined \cite{SS}. (For nonlinear problems, there exists the need to resort to numerical differentiation, even if analytic expressions of data are known, when the required order is high enough \cite{AC3}).

Recently, another technique has been suggested in \cite{ArP} in order to avoid order reduction by using rational methods (these include those associated to the stability function of Runge--Kutta methods). This technique does not require to impose any additional order condition nor does it require to resort to numerical differentiation of data either in space or in time. Nevertheless, this technique has just been applied and justified in \cite{ArP} for abstract linear initial value problems, but not for problems with time-dependent boundary conditions.

The aim of this paper is thus to suggest a technique to integrate linear initial boundary value problems (IBVPs) without order reduction starting from an $A$-stable rational approximation to the exponential, without requiring to impose any additional restrictions on the coefficients nor to use numerical differentiation of data. Moreover, the full discretization is also studied. To this end, an abstract framework which covers the usual spatial discretizations is considered.

%Although it would be very interesting to see, in particular problems, the comparison in terms of efficiency of this technique in contrast to %that in \cite{A} when analytic expressions of data are known, that will be a subject of future research since we believe that a careful and %complete study should be made considering also nonlinear problems. Moreover, the technique here is justified itself for the case in which %analytic expressions for data are not known.

The paper is structured as follows. Section 2 gives some preliminaries on the IBVPs. Then, some comments on the semigroup of translations and the way to deal with them in an approximated way is detailed in Section 3.  The lifted initial value problem which justifies the suggestion of the rational method is given in Section 4. In Section 5, the rational method applied to that lifted initial value problem is described and analysed. In Section 6, the suggested method for the time integration is completely given using the approximation for the part of the translations which was detailed in Section 3. In Section 7, an abstract framework for a general full discretization is described, as well as a thorough analysis of the method. Finally, in Section 8, some numerical experiments are shown which corroborate the previous results and a big gain in efficiency in comparison with the standard method of lines when using directly the Runge--Kutta method can be observed.

\section{The IBVP}
\label{SIBVP}
We briefly summarize some results in \cite{PA}. Let $X$ and $Y$ be  two complex Banach spaces and let us consider two linear operadors $A: D(A) \subset X \to X$ and $\partial: D(A) \subset X  \to Y$, such that $[A,\partial] ^T: D(A) \to X \times Y$  (i.e., $ [A,\partial] ^T(w) = [Aw,\partial w]^T$, for $w \in D(A)$) is closed.
Set
$$
D(A_0) = \mbox{Ker}(\partial)= \{ \, x \in D(A) \, : \, \partial x = 0 \, \}
$$
and let $A_0 : D(A_0) \subset X \to X$ be the restriction of $A$ to $D(A_0)$. In the context of PDEs, typically $X$ is an $L^q(\Omega)$ space, where $\Omega$ is a spatial domain with a piecewise regular boundary $\Gamma$, $A$ is a PD operator, and both $D(A)$ and   $Y$ are Sobolev spaces of functions defined on $\Omega$ and on $\Gamma$.   While the nature of $D(A)$ is induced by the PD operator $A$, the one  of $\partial$, that may contain or not derivatives (for instance, Neumann or Robin b.c. versus Dirichlet b.c.),  induces the adequate Sobolev space $Y$.

The basic hypotheses are:

\begin{enumerate}
\item[(H1)] $A_0 : D(A_0) \subset X \to X$ is the infinitesimal generator of a ${\cal C}_0$-semigroup  $\{S_{A_0}(t)\}_{t \ge 0}$ of linear and bounded operators in $X$. Therefore, $D(A_0)$ is dense in $X$ and there exist $M \ge 1$ and $\omega \in \R$ such that
$$
\| S_{A_0}(t) \| \le M \exp(\omega t), \qquad t \ge 0.
$$

\item[(H2)] There exists a bounded,  linear operator $E : Y \to X$, such that
$$
Ev \in D(A) \mbox{ and }  \partial E v = v, \qquad v \in Y,
$$
and $AE : Y \to X$ is  bounded. In the  PDEs' context,  $E$ is a linear  extension operator that, given $v$ in the Sobolev space $Y$, provides an element $w=Ev$ in the Sobolev space $D(A)$ such that $\partial w = v$.  The existence of $E$ is studied in the Extension Theory (see, e.g., \cite{Lions}).
\end{enumerate}

It turns out  \cite{PA} that, for $\Re(\lambda) > \omega$ and $v \in Y$, the eigenvalue problem
\begin{equation}
\label{eig}
\left\{
\begin{array}{lcl}
 A w &=& \lambda w,  \\
 \partial  w &=&v,
\end{array} \right.
\end{equation}
admits a unique solution, denoted by $w=K(\lambda) v$, that belongs to $D(A)$. This gives rise to a  bounded, linear operator $K(\lambda) : Y \to D(A)$, $\Re(\lambda) > \omega$,  that can be expressed, independently of the extension operator $E$, as
\begin{equation}
\label{extension}
K(\lambda) = [I - (\lambda I-A_0)^{-1} (\lambda I-A)] E, \quad \Re(\lambda)  > \omega.
\end{equation}
From this representation we readily obtain
$$
\| K(\lambda) \| \le \| E \| + \frac{|\lambda| \, \|E\| + \|AE \|}{\Re(\lambda)-\omega}, \quad \Re(\lambda) > \omega.
$$

The paper is mainly focused on the IBVPs
\begin{equation}
\label{problem}
\left\{
\begin{array}{lcl}
w' (t)&=& A w(t)+f(t), \quad  t\ge 0, \\
w(0)&=&w _0,  \\
\partial w(t)&=& g(t), \quad  t \ge 0,
\end{array} \right.
\end{equation}
with data $w_0 \in D(A)$, $f: [0,+\infty) \to X$ continuous, and $g:[0,+\infty) \to Y$ of class ${\cal C}^1$. In case this problem admits a genuine solution $w: [0,+\infty) \to D(A)$, we can justify that,
\begin{equation}
\label{necregD}
g'(t) = \partial w'(t)= \partial (A w(t) + f(t) ) , \qquad t \ge 0.
\end{equation}
In particular, for $t=0$, we see that the data must satisfy the so called natural compatibility condition
\begin{equation}
\label{necreg0D}
g'(0) = \partial (A w_0 +   f(0)),
\end{equation}
that turns to be  a necessary one  to have a ${\cal C}^1$ solution of (\ref{problem}). Higher regularity imposes more natural compatibility conditions on data. It can be proved \cite{PA} that (\ref{necreg0D}) is also a sufficient condition,  i.e., under it (\ref{problem}) admits a unique genuine solution $w$ and then  (\ref{necregD}) remains valid for $t \ge 0$. Moreover,  the IBVP (\ref{problem}) is well posed \cite{PA} in the sense that the genuine solution $w$ depends continuously on  $w_0 \in D(A)$, $f  \in {\cal C}([0,+\infty),X)$, $g \in {\cal C}^1([0,+\infty),Y)$,  (for $f$ we consider the $L_1$ norm and for $g$ the total variation one), so that we can consider generalized solutions of (\ref{problem}) in the framework of $X \times L^1_{loc}([0,+\infty),X) \times BV_{loc}([0,+\infty),Y)$.

It is clear that, for $\Re(\lambda) > \omega$, the solution of (\ref{problem}) can be expressed as
$$
w(t) = u(t)+ K(\lambda ) g(t), \qquad t \ge 0,
$$
where, since  $\partial  u (t) = 0$, $t \ge 0$, and $A K(\lambda) = \lambda K(\lambda)$, the term $u: [0,\infty) \to X$ solves
$$
\left\{
\begin{array}{lcl}
u'(t) &=& A_0  u(t) + \lambda K(\lambda) g(t) + f(t) -K(\lambda) g'(t), \qquad t \ge 0, \\
u(0) &=& w_0-K(\lambda) g(0),
\end{array} \right.
$$
so that,  when $\omega <0$,  the choice $\lambda=0$ results in the simpler IVP
\begin{equation}
\label{equdotu}
\left\{
\begin{array}{lcl}
u'(t) &=& A_0  u(t) + f(t) -K(0) g'(t), \qquad t \ge 0, \\
u(0) &=& w_0-K(0) g(0).
\end{array} \right.
\end{equation}
Alternatively, the IBVP (\ref{problem}) can also be reduced to an IVP by using an available extension operator $E$, instead of $K(0)$, but then we must use the source term $f+AEg-Eg'$.

Let us notice that we can easily  reduce the problem to the situation $\omega < 0$. To this end, we just select $\alpha >  \omega$ and express the solution $w$ of the IBVP (\ref{problem}) in the form
$$
w(t) = \mbox{e}^{t \alpha} w_\alpha(t), \qquad t \ge 0,
$$
where $w_\alpha: [0,+\infty) \to X$ is the solution of the conjugate problem
$$
\left\{
\begin{array}{lcl}
w_\alpha'(t)&=& (A-\alpha I) w_\alpha (t)+\mbox{e}^{-t \alpha} f(t), \quad  t\ge 0, \\
w_\alpha(0)&=&w_0,  \\
\partial w_\alpha(t)&=& \mbox{e}^{-t \alpha} g(t), \quad  t \ge 0.
\end{array} \right.
$$
The generator of the above problem is $A_0 - \alpha I$ whose spectral abcissa is $\omega -  \alpha < 0$. This reduction simplifies the presentation and it is also interesting from the numerical point of view. On the one side, no restriction on the used step size is required and, on the other, the basic estimates in \cite{BT} involving rational approximations are simpler when $\omega \le 0$. Actually, when $\omega <0$ we can conjugate with $0 < \alpha < |\omega|$ and in this way introduce an exponential damping in the  mentioned basic estimates. Thus, in the rest of the paper, we will make the simplifying assumption that $\omega < 0$ and consider the IVP (\ref{equdotu}).

In Section \ref{SFD} we treat the full discretization of (\ref{problem}) and, to deal with the spatial consistency,  we will introduce  two Banach spaces $(W, \| . \|_W)$ and $(Z, \| \cdot \|_Z)$, continuously embedded in $X$ and $Y$, such that
\begin{equation}
\label{WZ}
W \subset D(A) \quad \mbox{\rm and} \quad K(0)  Z \subset W.
\end{equation}
Under these conditions, the restriction of $K(0)$, resp. $\partial$, to $Z$, resp. to $W$, are continuous from $Z$ to $W$, resp., from $W$ to $Z$. In the common applications, $W$ is and $Z$ are Sobolev spaces,  with norms finer than those of $X$ and $Y$.  The Extension Theory is the tool to provide the existence of an extension operator  $E: Z \to W$ such that $\partial E = I$. Once $E$ is obtained,
(\ref{extension}) shows that $K(0)\partial $ leaves $W$ invariant, as soon as $A_0^{-1}(A W) \subset W$, that is the usual situation.

It is important to remark that, for boundary data $g:[0,+\infty) \to Y$ of class ${\cal C}^1$, the IBVP (\ref{problem}) also makes sense in the framework of $W$. Actually, (\ref{problem}) admits a genuine solution $w: [0,+\infty) \to W$ if, and only if, the natural compatibility condition (\ref{necreg0D}) is satisfied, in which case (\ref{necregD}) remains valid. In fact, if we set $W_0 = W \cap \mbox{\rm Ker}(\partial)$, endowed with norm $\| \cdot \|_{W_0}$ induced by $\| \cdot \|_W$, (\ref{equdotu}) is a standard, non-homogeneous, linear problem in $W_0$. Besides, for $t \ge 0$, it is clear that
\begin{eqnarray}
\label{cotaW0u}
\| u(t) \|_{W_0} \le \| w(t) \|_W + \| K(0) g(t) \| \le  \| w(t) \|_W + \| K(0) \partial \|_{W \to W} \| w(t) \|_W,
\end{eqnarray}
an important estimate when considering the spatial discretization. Notice also that this estimate implies that the restriction of $S_{A_0}(t)$, $t \ge 0$, to $W_0$, forms a ${\cal C}_0$-semigroup on $W_0$, whose infinitesimal generator is the restriction of  $A_0$ to $W_0$.

In the first sections of the manuscript we study the time discretization of  IVPs with the more general format
\begin{equation}
\label{equdot}
\left\{
\begin{array}{lcl}
u'(t)&=& A_0 u(t)+f(t)-Kg'(t), \quad t \ge 0, \\
u(0)&=& u_0:=w_0-Kg(0),
\end{array}
\right.
\end{equation}
where $K : Y \to X$ a bounded, linear operator. The situation $K = K(0)$ is already motivated  and the wider format of IVP (\ref{equdot}) may occur  in other applications.

\section{The semigroup of translations}
\label{SST}
In this section $(Z, \| \cdot \|)$ stands for a general complex Banach (not the one which has been introduced in the previous section).  For $m \ge 0$, let ${\cal C}_{ub}^m([0,+\infty),Z)$ denote the space formed by all the ${\cal C}^k$ mappings $h : [0,+\infty) \to Z$ such that  $h^{(j)}$, $0 \le j \le m$, are bounded and uniformly continuous on $[0,\infty)$.
Set, for $m \ge 0$ and $0 \le t \le +\infty$,
\begin{equation}
\label{defnorm}
\| h \|_{m, t} = \max_{0 \le j \le m} \sup _{0 \le s \le t} \| h^{(j)} (s) \|.
\end{equation}
The space ${\cal C}_{ub}^m([0,+\infty),Z)$, endowed with the norm $\| \cdot \|_{m,\infty}$, is a Banach space.

For $t \ge 0$,  let $T(t) : {\cal C}_{ub}([0,+\infty),Z) \to {\cal C}_{ub}([0,+\infty),Z)$ be the shift operator
$$
[T(t) h] (s) = h(t+s), \qquad h \in {\cal C}_{ub}([0,+\infty),Z), \quad s \ge 0.
$$
The familly $\{ T(t) \}_{t \ge 0}$ is a strongly continuous semigroup of contractions on  ${\cal C}_{ub}([0,+\infty),Z)$ whose generator $B$ (see, e.g., Lemma 4.1 in \cite{ArP}) acts on the domain
$B : D(B) =  {\cal C}_{ub}^1([0,+\infty),Z)  $ as $B h = h'$, for  $h \in D(B)$.
Therefore,  the restrictions of $T(t) $, $ t \ge 0$, to  $D(B) =  {\cal C}_{ub}^1([0,+\infty),Z) $ forms another strongly continuous semigroup, whose generator is the restriction of $B$ to $D(B^2) = {\cal C}_{ub}^2([0,+\infty),Z)$.

The procedure that we will suggest in Section~\ref{SM} relies on Lemma~4.3 in  \cite{ArP} that, for the convenience of both the reader and  presentation, we comment next.

By Hille-Yosida Theorem, the spectrum of $B$ is contained in the half plane $\Re(z) \le 0$. Thus see, e.g. \cite{BT} and (\ref{rz}) below, given a rational mapping $R(z)$ without poles on such a half plane,  it makes sense to consider the linear and bounded operators  $R(\tau B)$, $ \tau >0$, on $ {\cal C}_{ub}([0,+\infty),Z)$. On the other hand, let $L :  {\cal C}_{ub}([0,+\infty),Z) \to Z$ denote the delta operator
$$
L h =h(0), \qquad h \in {\cal C}_{ub}([0,+\infty),Z),
$$
so that $ h(t) = h(t+s)|_{s=0}=L T(t) h $, for $t \ge 0$ .

The question that arises in Section~\ref{SM} is how to accurately compute expressions of the form
\begin{equation}
\label{goal}
L R(\tau B) T(t)h, \qquad \tau > 0, \quad t \ge 0,
\end{equation}
by using only evaluations of $h \in {\cal C}_{ub}([0,+\infty),Z)$. After developing $R(z)$  into simple fractions (see, e.g., (\ref{rz})), what we need is to solve a certain number of ODEs of the form
\begin{equation}
\label{edo}
\phi (t)- \tau w \phi'(t) = \psi (t), \qquad t \ge 0,
\end{equation}
with $\Re(w) >0$ and datum $\psi \in {\cal C}_{ub}([0,+\infty),Z)$, under the additional condition that $\phi$ must belong to ${\cal C}_{ub}([0,+\infty),Z)$, that is not  a non-trivial computational task. Let us notice that only $\phi(0)$ would be  required in (\ref{goal}). Actually, overcoming this difficulty is the aim of Lemma~4.3 in \cite{ArP}.

We first introduce some notation. For $\b v \in \R^m$, we use the  standard notation $\b v \ge 0$ to indicate that all the components of $\b v$ are $\ge 0$. Given $\b c \in \R^q$, $q \ge 1$, $v : [0,+\infty) \to Z$, $t \ge 0$,  $\tau >0$ such that $t+\tau \b c \ge 0$,  and a mapping $h : [0,+\infty) \to Z$, we set $h(t+\tau \b c)=[h(t + \tau c_1), h(t+\tau c_2), \ldots, h(t+\tau c_q)]^T \in Z^q$. Notice that we do not assume that $\b c \ge 0$. Moreover, for a vector $\b \gamma \in \C^q$, we set $\b \gamma^T \cdot h(t+\tau \b c) =\sum_{i=1}^q \gamma_i h(t +\tau c_i) \in Z$.  Finally,  $\R^q_d$ stands for  the set formed by all the vectors $\b c \in \R^q$ with $q$ different components, i.e., such that $c_i \ne c_j$, for $i \ne j$, $1 \le i, \, j \le q$.

Let us fix $t >0$, $\tau >0$, $q \ge 1$,  and select $\b c \in \R^q_d$ such that $t +\tau c_i \ge 0$, $1 \le i \le q$.
Then we consider the Taylor expansion $R(z) = R_0 + R_1 z + \cdot + R_{q-1} z^{q-1} + O(z^q)$, up to a given order $q$. It is clear that the Vandermonde system
\begin{equation}
\label{Vander}
c_1^j \gamma_1+\dots+c_q^j \gamma_q= j! \,R_j, \quad 0 \le j \le q-1,
\end{equation}
admits a unique solution $\b \gamma \in \C^q$.  The forenamed Lemma~4.3 in \cite{ArP} (based on the results in \cite{BT}) asserts that
$$
\|R(\tau B)h(t+\cdot)-\b \gamma^T h(\cdot+\tau \mathbf{c})\|\le \kappa \tau^q \|B^q h\|, \qquad h \in D(B^q),
$$
for some $\kappa >0$. We notice that both $\b \gamma$ and $\kappa$ depend on $R(z)$ and $\b c$, but not on $\tau$.  Once we have obtained $\b \gamma$, we have
$$
LR(\tau B)T(t)  h =L R(\tau B) h(\cdot + t) \approx L \b \gamma ^T h(\cdot+\tau \b c)=
\b \gamma^T h(\tau \b c),
$$
so that
\begin{equation}
\label{lema43}
\| LR(\tau B)T(t)  h - \b \gamma ^T h(\tau \b c) \| \le \kappa \tau^q \|B^q h\|, \qquad h \in D(B^q),
\end{equation}
meaning that $ \b \gamma^T h(t+\tau \b c)$ is an approximation to (\ref{goal})  of order $q$ that uses $q$ evaluations of $h$, as desired.

Whereas for linear,  non-homogeneous problems ($g =0$) in (\ref{problem}), the above semigroup was used in \cite{ArP} with $Z=X$, in the present context of (\ref{equdot}) we will also use it with $Z=Y$.

\begin{remark}
\label{kappa}
The behaviour of $\kappa$ is  related to the particular choice of the nodes $\b c$. In \cite{AP, Tesis}, a more in-depth study of this influence is carried out. In particular, there it is proved that if $M = M(\b c)$ is the coefficients matrix in linear system (\ref{Vander}) with the rows scaled by $j!$,  $\kappa$ can be taken as
$$
C+\frac{1}{(q-1)!}\|M^{-1}\|_{\infty} \|[R_0,\dots,R_{q-1}]^T\|_{\infty} \max_{1 \leq k \leq q}  \left|c_k\right|^q,
$$
where the constant $C$ is independent of  $\b c$.

On the one hand, it can be checked that $\rho=\|M^{-1}\|_{\infty}$ at the nodes $\b c=[-q+1,-q+2,\dots,0]$ take the values $[3,6.5,13.25,32.25,76.94,199.04,511.31]$ for $3 \le q \le 9$, while at the $q$ centered nodes $-(q-1)/2+[0:q-1]^T$, these values are $[2, 2.94, 3.75, 5.78, 8.03, 13.15, 20.27]$ (c.f. \cite{G}). Thus, for $q\le 7$, even with the first choice, the calculation of $\b \gamma$ in (\ref{Vander}) can be made without introducing big round-off errors.
%$[2,\, 5  ,\,9.5  ,\, 16.7  ,\,37.4  ,\,133.5 ,\,  356]$, moderate enough for practical use at least for $q \le 7$. Besides, with symmetric equispaced  nodes, the terrible behaviour of the Lebesgue constants (of order $2^q$), does not affect the evaluations at the central point, since,  again for $3 \le q \le 9$, the values of the Lebesgue function at such points are  $[1 ,\,  1 ,\,  1.25,    1,\,   1.4,\,    1,\, 1.5]$ (it is clear that for an odd value of $q$ the value is $1$, because the center is then a nodal point). In summary, such  equispace,  symmetric  choice looks well suited to practical purposes, for $q \le 7$.}

On the other hand, the factor $\max_{1 \leq k \leq p} \left|c_k\right|^p$ suggests that, among all equispaced nodes, the most central ones are those which minimize $\kappa$. Chebyshev nodes also have good properties in this point but gives rise to a more expensive method per step as a whole, as stated afterwords in Remark \ref{naturalchoice}.
\end{remark}

\section{The lifted IVP}
\label{SL}
Following the approach in \cite{ArP}, given a linear and bounded operator $K : Y \to X$, we embed the non homogeneous IVP (\ref{equdot}) into an enlarged, homogeneous one. To this end, we consider both the semigroup of translations on  ${\cal C}_{ub}([0,+\infty),X)$ and on ${\cal C}_{ub}([0,+\infty),Y)$, that are denoted by $T_0(t)$ and $T_1(t)$, $t \ge 0$, respectively. Accordingly, their generators are denoted by $B_0$ and $B_1$. Recall that $D(B_0^m) = {\cal C}_{ub}^m([0,+\infty),X)$, $D(B_1^m) = {\cal C}_{ub}^m([0,+\infty),Y)$, for $ m \ge 0$.

Let $L_0 : {\cal C}_{ub}([0,+\infty),X)  \to X$ and $L_1 : {\cal C}^1_{ub}([0,+\infty),Y)  \to Y$ be the delta operators
$$
L_0 f = f(0), \, \quad f \in  {\cal C}_{ub}([0,+\infty),X); \qquad L_1 g =g(0), \quad g \in  {\cal C}^1_{ub}([0,+\infty),Y).
$$

The product $H = X \times  {\cal C}_{ub}([0,+\infty),X) \times {\cal C}^1_{ub}([0,+\infty),Y)$, endowed with the norm
$$
\| [u,f,g]^T \|=  \| u \| + \| f \|_\infty+ \max\{ \| g \|_\infty, \| B_1 g \|_\infty \}, \qquad [u,f,g]^T \in H,
$$
is a Banach space.  On the domain $D(G) = D(A_0) \times D(B_0) \times D(B_1^2) \subset H$, let us define  the operator $G: D(G) \subset H \to H$ by
$$
G= \left( \begin{array}{ccc} A_0 & L_0 & -K L_1B_1\\ 0 & B_0 & 0 \\ 0 & 0 & B_1 \end{array}\right),
$$
and consider the linear, homogeneous IVP on $H$
\begin{equation}
\label{hom}
\left\{ \begin{array}{lcl}
U'(t) &=& GU(t) \qquad t \ge 0, \\
U(0) & = & U_0\in D(G).
\end{array} \right.
\end{equation}
Writing $U(0)=[u_0,f,g]^T$,  $U(t) = [u(t),\phi(t),\psi(t)]^T \in H$, $t \ge 0$, the last two components  trivially yield
$$
\phi(t) = T_0 (t) f= f(t+\cdot) , \qquad \psi (t) = T_1(t) g=g(t+\cdot),
$$
while the first one fits into the equation
$$
u'(t) = A_0 u(t) + f(t) - K g'(t), \qquad t \ge 0.
$$
These remarks show readily that (\ref{hom}) admits a unique genuine solution that, by using  the  variation-of-constants formula, can be represented as
$
U(t) = S_G(t) U(0), \, t \ge 0,
$
where $S_G(t) : H \to H$, $t \ge 0$, is the linear operator
$$
S_G(t)=\left( \begin{array}{ccc}S_{A_0}(t) &\,\, \int_0^t  S_{A_0}(t-s) L_0 T_0(s)\, \mbox{d} s\,\, &
\,\,-\int_0^t  S_{A_0}(t-s) K L_1 B_1T_1(s) \,\mbox{d} s \\
0 & T_0(t) & 0 \\
0 & 0 & T_1(t) \end{array} \right) ,$$
(the integrals are understood in the strong sense). Clearly, $S_G(t)$, $t \ge 0$, is a strongly continuous semigroup on $H$ and
\begin{eqnarray}
\|S_G(t)\|\le M(1+2t).
\label{boundSG}
\end{eqnarray}

It will be useful to introduce the family of seminorms $||| \cdot |||_{m,t}$ in the product space ${\cal C}_{ub}^m([0,+\infty),X) \times  {\cal C}_{ub}^m([0,+\infty),X) \times  {\cal C}_{ub}^{m+1}([0,+\infty),Y)$ given, for $[v,\phi,\psi]^T$ in such a product, by the expression
\begin{equation}
\label{defseminorm}
||| [v,\phi,\psi]^T |||_{m,t} =\| v\|_{m,t} + \| \phi\|_{m,t}+  \| \psi\|_{m+1,t}
\end{equation}
($\| \cdot\|_{m,t}$ is defined in (\ref{defnorm})).

Let us stress that, for $m \ge 1$,  $U \in D(G^{m})$ if, and only if, $S_G(\cdot)U_0 \in {\cal C}_{ub}^{m}([0,+\infty),X)$. This is equivalent to have $u \in {\cal C}_{ub}^{m}([0,+\infty),X)$,  $f \in {\cal C}_{ub}^{m}([0,+\infty),X)$ and $g\in {\cal C}_{ub}^{m+1}([0,+\infty),Y)$. Therefore, under such smoothness conditions on $u$, $f$ and $g$, the solution $U $ of (\ref{hom}) takes values in $D(G^m)$ and, since $G^m U(s) =   U^{(m)}(s) $, for $s \ge 0$, it turns out that
\begin{equation}
\label{seminorm}
\| G^m U (s) \| = \| U^{(m)}(s)  \| \le ||| U |||_{m,t}, \qquad 0 \le s \le t.
\end{equation}

Notice that, for $m \ge 1$,  $D(G^m) \subset D(A_0^m) \times {\cal C}_{ub}^{m}([0,+\infty),X) \times  {\cal C}_{ub}^{m+1}([0,+\infty),Y)$, but  the equality is only true for $m =1$. Actually,  to guarante that  $U \in {\cal C}_{ub}^{m}([0,+\infty),X)$ we must  impose several compatibility conditions on the initial data.

\section{Rational methods for the lifted IVP}
\label{SRML}
Henceforth,  $r(z)$ will denote  an A-aceptable approximation to the exponential of order $p \ge 1$, that is,
$| r(z) | \le 1$, for $ \Re(z) \le 0$, and $r(z) - \mbox{e}^z = O(z^{p+1})$, as $z \to 0$.

The expansion of $r(z)$ into simple fractions is of the form
\begin{eqnarray}
r(z)=r_{\infty}+\sum_{l=1}^k \sum_{j=1}^{m_l} \frac{r_{lj}}{(1-z w_l)^j},
\label{rz}
\end{eqnarray}
for certain complex coeficcients $r_\infty$, $r_{lj}$, $1 \le l \le k$, $1 \le j \le m_l$, and values $w_l\in \mathbb{C}$ with $\mbox{Re}(w_l)>0$, for $1\le l \le k$.  The number of poles of $r(z)$ (including its multiplicity) is $s=\sum_{1 \le l \le k} m_l$.

Since we ssume that $\omega <0$, Hille--Phillips Theorem implies that the linear operator
\begin{equation}
\label{ratfun}
r(\tau G)
 = r_\infty I + \sum_{l=1}^k \sum_{j=1}^{m_l} r_{lj}(I-\tau w_lG)^{-j}, \qquad \tau >0,
\end{equation}
is well defined and uniformly bounded for $\tau >0$.

Given an initial value $U_0 \in H$ and $\tau >0$, the recurrence
\begin{equation}
\label{ratmet}
U_{n+1} = r(\tau G) U_n, \qquad n \ge 1,
\end{equation}
defines  the numerical approximation  $U_n \in H$ to  $U(t_n)$,  at $t_n = n \tau$, by means of the   rational method based on $r(z)$.

Let us explore the form of $r(\tau G)$. For $\Re (z) >0$, we check that
$$
(I -z G)^{-1} = \left( \begin{array}{ccc}(  I -zA_0)^{-1} & zQ_{0,1}(z) &-\,Q_{1,1}(z) \\ 0 & ( I - zB_0)^{-1}& 0 \\ 0 & 0 & (I -z B_1)^{-1} \end{array}\right),
$$
where
$$
Q_{0,1}(z)= (I-zA_0)^{-1} L_0(I-zB_0)^{-1} ,\quad  Q_{1,1}(z) =  (I-z A_0)^{-1} K L_1 z B_1 (I-zB_1)^{-1} .
$$
For $j \ge 1$, let us set
$$
(I -z G)^{-j} = \left( \begin{array}{ccc}(  I -zA_0)^{-j} & z Q_{0,j}(z) &-\,Q_{1,j}(z) \\ 0 & ( I - zB_0)^{-j}& 0 \\ 0 & 0 & (I -z B_1)^{-j} \end{array}\right), \qquad j \ge 1
$$
for suitable operators $Q_{m,j}(z)$, $m =0,\, 1$.  By induction, using the variation-of-constants formula, it is easy to conclude that, for $j \ge 1$,
\begin{equation}
\label{powersresol}
\left\{ \begin{array}{lcl}
Q_{0,j}(z) &=&\sum_{i=1}^j (I-zA_0)^{-j+i-1}L_0(I-zB_0)^{-i} ,\\[3pt]
Q_{1,j}(z) &=&  \,\sum_{i=1}^j (I-zA_0)^{-j+i-1}K L_1  z B_{1}(I-zB_{1})^{-i} .
\end{array} \right.
\end{equation}

This result, used  in (\ref{ratmet}) with $z=\tau w_l$, $1 \le l \le k$, $1 \le j \le m_l$,  readily yields
\begin{equation}
\label{expresionr}
r(\tau G) = \left( \begin{array}{ccc}r(\tau A_0) & \tau E(\tau) & -\,F(\tau) \\ 0 & r(\tau B_0) & 0 \\ 0 & 0 & r(\tau B_1) \end{array}\right),
\end{equation}
where
\begin{eqnarray}
\label{E}
E(\tau)&=&\sum_{l=1}^k \sum_{j=1}^{m_l} r_{lj} w_l \sum_{i=1}^j (I- \tau w_l A_0)^{-j+i-1}L_0  (I- \tau w_l B_0)^{-i}, \\
\label{F}
F(\tau)&=&\sum_{l=1}^k \sum_{j=1}^{m_l} r_{lj} w_l \sum_{i=1}^j (I- \tau w_l A_0)^{-j+i-1} K(0) L_1 \tau B_1 (I- \tau w_l B_1)^{-i}.
\end{eqnarray}

In summary, for a given initial datum $U_0 =[u_0,f,g]^T \in H$, the rational method (\ref{ratmet}) generates approximations $U_n = [u_n,f_n,g_n]^T$, $n \ge 0$, with $f_n = r(\tau B_0)^n f$, $g_n = r(\tau B_1)^n g$ and the first components are provided by the recurrence
\begin{equation}
\label{methodrat}
u_{n+1} = r(\tau A_0) u_n + \tau E(\tau)r(\tau B_0)^n f - F(\tau)r(\tau B_1)^n g, \qquad n \ge 0,
\end{equation}
Thus, one step requires solving $s:=\sum_{l=1}^k m_l$ linear systems with the different operators $(I-\tau w_l A_0)$, $1 \le l \le k$, something that we can assume to be numerically affordable whithin some tolerance, but, as we commented, it  also requires solving linear systems with the operators $(I-\tau w_l B_0)^{-1}$ and $(I-\tau w_l B_1)^{-1}$, a difficulty we will avoid by using (\ref{lema43}).

Concerning the stability, since $S_G$ contains components which are semigroups of translations, we can face the worst case \cite{BT} and it may happen that
$$
\| r(\tau G )^n \| \ge c \sqrt{n}, \qquad n \ge 0,
$$
for some $c >0$.  However, we will consider only perturbations of the first component and, to this purpose, what matters is the stability of $r(\tau A_0)$.  In general \cite{BT}, we have
$$
\| r(\tau A)^n \| \le CM n^{\alpha} \qquad n \ge 0, \tau \ge 0,
$$
where $C >0$ and $0 \le \alpha \le 1/2$ are constants that depend exclusively on $r(z)$. It is known that $\alpha = 0$ in case $r(z)$ is a Pad\'e approximation to the exponential. Moreover, in the most interesting instances, either when $A_0$ is a dissipative operator on a Hilbert space or when $S_{A_0}$ is a holomorphic semigroup, it turns out that
\begin{equation}
\label{Cs}
C_s(r,A_0):= \sup_{\tau > 0} \| r(\tau A_0)^n \| < +\infty
\end{equation}
Henceforth, by  simplicity, we will assume that $C_s(r,A_0) < +\infty$ and set $C_s=C_s(r,A_0)$. The results in the paper remain true   just by putting $C_s(n,r,A_0)$ instead of $C_s$, meaning that there are cases for which we must introduce an extra factor $n^\alpha$ in the estimates.

We are mainly  interested in consistency.  In the rest of the section we will assume that $u \in {\cal C}_{ub}^{p+1}([0,\infty),X)$, $f \in {\cal C}_{ub}^{p+1}([0,\infty),X)$ and $g \in {\cal C}_{ub}^{p+2}([0,\infty),Y)$.

Theorem~3 in \cite{BT}  shows that for $U_0 \in D(G^{p+1})$,
$$
\|S_G(t_{n+1})U_0-r(\tau G)S_G(t_n)U_0 \| \le C_{el} M \tau^{p+1}(1+ 2t_n)\| G^{p+1}U_0\|.
$$
We conclude that for smooth solutions (recall  (\ref{defseminorm}) and (\ref{seminorm})), we have
\begin{equation}
\label{cotahomlocal}
\| S_G(t_{n+1})U_0 - r(\tau G)S_G(t_n)U_0  \| \le C_{el}M \tau^{p+1}(1+ 2t_n) \, ||| [u,f,g]^T |||_{p+1,t_n}.
\end{equation}
The  first component $\rho_n \in X$ of the local error $S_G(t_{n+1})U_0-r(\tau G)S_G(t_n)U_0$, $n \ge 0$,
satisfies
\begin{equation}
\label{deflocal}
u(t_{n+1}) = r(\tau A_0) u(t_n) + \tau E(\tau) T_0(t_n)f- F(\tau) T_1(t_n) g+\rho_n,
\end{equation}
and (\ref{cotahomlocal}) implies that
\begin{equation}
\label{localrat}
\| \rho_n\| \le C_{el} M  \tau^{p+1}(1+ 2t_n) \, ||| [u,f,g]^T |||_{p+1,t_n}, \qquad n \ge 0.
\end{equation}

\begin{remark}
\label{remarkrmpv}
Variable step-sizes $\tau_n >0$, $n \ge 1$, can be considered for rational methods. In our context, it would be necessary to account for the stability constant $ C_s=\sup_{n \ge 1} \| \prod_{j=0}^n r(\tau \tau_j) \|$.
It is known that $C_s < +\infty$ when $S_{A_0}$ is a holomorphic semigroup and $r(z)$ is strongly stable. Moreover, in this situation, A($\theta$)-acceptable rational mappings can be considered for suitable angles $\theta$ related to the angle of the sector where  $S_{A_0}$ is holomorphic. For general semigroups, the techniques in \cite{BT} can show that $C_s < +\infty$ if $a \le \tau_n/\tau_m \le b$, for $n, \,m \ge 0$, and some $0 <a <b$.
\end{remark}

\section{The suggested method}
\label{SM}
As we mentioned in Section~\ref{SST}, the practical difficulty of the rational method (\ref{ratmet}) is evaluating the different expressions $L_0(I-\tau w_lB_0)^{-i}$ and  $L_1 \tau B_1(I-\tau w_lB_1)^{-i}$, $ 1 \le l \le k$,
$1 \le i \le m_l$, ocurring in (\ref{E}) and (\ref{F}), and we  have already suggested a way to overcome it (\ref{lema43}).   The difficulty is also present in the evaluations of $r(\tau B_0)$ and $r(\tau B_1)$, but, as we will show soon, these operators can be substituted by  $T_0(\tau)$ and $T_1(\tau)$.

Given $\tau >0$, we select two auxiliar  sequences $\{ \b c_n \}_{n=0}^{+\infty}$ and $\{ \b d_n \}_{n=0}^{+\infty}$, in $ \R^p_d$ and $\b d_n \in  \R^{p+1}_d$, resp.,
such that $t_n + \tau \b c_n \ge 0 $ and $t_n + \b d_n \ge 0$, $n \ge 0$.  Thus, in principle, we will have as many versions of the method we suggest as many possible choices of the auxiliar sequences.

Let us fix $1 \le l \le k$ and $1 \le i \le m_l$. For each $n \ge 0$, we solve the Vandermonde system (\ref{Vander}) corresponding to the nodes $\b c_n$, resp.   $\b d_n$, and the rational mappings $(1-w_lz)^{-i}$ , resp.  $z(1-w_l z)^{-i}$. This yields vectors $\b \gamma_{n,l,i} \in \C^p$ and $\b \eta_{n,l,i} \in \C^{p+1}$
such that, by (\ref{lema43}), satisfy

\begin{equation}
\label{cotaaux1}
\| L_0(I-\tau w_lB_0)^{-i} f(t_n + \cdot)  - \b \gamma_{n,l,i}^T \cdot  f(t_n+\tau \b c_n) \| \le \kappa_n \tau^{p} \| f^{(p)} \|_\infty
\end{equation}
and
\begin{equation}
\label{cotaaux2}
\| L_1 \tau B_1 (I-\tau w_lB_1)^{-i} g(t_n + \cdot)  - \b \eta_{n,l,i}^T \cdot  g(t_n+\tau \b d_n) \| \le \kappa_n \tau^{p+1} \| g^{(p+1)} \|_\infty,
\end{equation}
for some $\kappa_n =\kappa(\b c_n, \b d_n)$.

The basic idea is to introduce the operators $E_n(\tau) : {\cal C}_{ub}([0,+\infty),X) \to X$ and $F_n(\tau) : {\cal C}_{ub}([0,+\infty),Y) \to X$ defined by
\begin{equation}
\label{En}
E_n(\tau)f=\sum_{l=1}^k \sum_{j=1}^{m_l} r_{lj} w_l \sum_{i=1}^j (I- \tau w_l A_0)^{-j+i-1}\b \gamma_{n,l,i}^T \cdot  f(t_n+\tau \b c_n),
\end{equation}
\begin{equation}
\label{Fn}
F_n(\tau)g=\sum_{l=1}^k \sum_{j=1}^{m_l} r_{lj} w_l \sum_{i=1}^j (I- \tau w_l A_0)^{-j+i-1} K\b \eta_{n,l,i}^T \cdot  g(t_n+\tau \b d_n).
\end{equation}
It is clear, by Hille--Phillips Theorem and (\ref{cotaaux1}) and (\ref{cotaaux2}), that there exists $C_{aux} >0$, depending only on $r(z)$, such that
\begin{equation}
\label{Delta1}
\| \tau E(\tau)f - \tau E_n(\tau)f \| \le \kappa_n C_{aux} M \tau^{p+1} \| f^{(p)} \|_\infty, \quad f \in {\cal C}_{ub}^{p}([0,+\infty),X),
\end{equation}
and
\begin{equation}
\label{Delta2}
\| F(\tau)g - F_n(\tau)g \| \le  \kappa_n C_{aux} M \| K \| \tau^{p+1} \| g^{(p+1)} \|_\infty, \quad g  \in {\cal C}_{ub}^{p+1}([0,+\infty),X),
\end{equation}

These facts suggest to modify the rational method (\ref{methodrat}) and, instead of (\ref{E}) and (\ref{F}), use the operators $E_n(\tau)$ and $F_n(\tau)$. Moreover, we also suggest to use directly $T_0(t_n) f$ and $KT_1(t_n)$ in (\ref{methodrat}).

Thus, let $u \in {\cal C}([0,+\infty),X)$ be the solution of (\ref{equdot}) with  data $u_0 \in X$, $f \in {\cal C}_{ub}([0,+\infty),X)$, $g \in {\cal C}_{ub}^1([0,+\infty),Y)$. For $\tau >0$, we propose the recurrence
\begin{equation}
\label{method}
\bar u_{n+1} = r(\tau A_0) \bar u_n + \tau E_n(\tau) f(t_n + \cdot) - F_n(\tau)g(t_n + \cdot), \qquad n \ge 0,
\end{equation}
with $\bar u_0 = u_0$, as the numerical procedure to time integrate (\ref{equdot}). This  procedure avoids the practical difficulty of solving systems in (\ref{ratmet}), avoids the use of $r(\tau B_0)$ and $r(\tau B_1)$, and it only requires evaluations of $f$ and $g$. In  the sequel, the suggested procedure will be called a rational like (RL) method. The specification of the version requires a full notation  $RL(r(z), \tau,\{ \b c_n \}_{n=0}^{+\infty}, \{ \b d_n \}_{n=0}^{+\infty})$. However, once the parameters are fixed,   we will simply talk about the RL method.

The analysis of the convergence of (\ref{method}) is rather simple and it is carried out by the classical approach of stability and consistency. The proof in the present manuscript is somehow simpler that the related one in \cite{ArP}. The reason is that the presentation in \cite{ArP} was intended to cope, in a simultaneous way, with analytic semigroups and data. Let us recall (\ref{defseminorm}) and (\ref{Cs}).

\begin{lemma}
\label{stability}
Let $\delta_n \in X$, $n \ge 0$, be a sequence of perturbations and let $\bar u_n^* \in X$, $n \ge 0$, be the solution of the perturbed recurrence
$$
\bar u_{n+1}^* = r(\tau A_0) \bar u_n^* +
\tau E_n(\tau) f(t_n + \cdot) - F_n(\tau)g(t_n + \cdot)+\delta_n, \qquad n \ge 0,
$$
with initial value $\bar u_0^* = \bar u_0 +\delta_0$. Then,
$$
\| \bar  u_n^* -\bar   u_n \| \le C_s \sum_{j=0}^{n-1} \| \delta_j \|, \qquad n \ge 0.
$$
\end{lemma}

\begin{proof} Since the difference  $\bar u_n^* - \bar u_n$, $n \ge 0$, satisfies
$$
\bar u_{n+1}^* - \bar u_{n+1} = r(\tau A_0) (\bar u_n^* - \bar u_n)+\delta_n, \qquad n\ge 0,
$$
the proof is trivial (note that $C_s \ge 1$). \end{proof}

The local error of the RL method, at time level $n \ge 0$, is now the residual $\bar \rho_n \in X$ satisfying
\begin{equation}
\label{deflocalsm}
u(t_{n+1})= r(\tau A_0) u(t_n) + \tau E_n(\tau) T_0(t_n) f - F_n(\tau) T_1(t_n) g +\bar \rho_n.
\end{equation}
The consistency is provided by the next lemma.

\begin{lemma}
\label{consistence}
Let  us assume that $u \in {\cal C}^{p+1}([0,+\infty),X)$, $f \in {\cal C}^{p+1}([0,+\infty),X)$ and $g \in {\cal C}^{p+2}([0,+\infty),Y)$. Then, the corresponding local errors $\bar \rho_n$, $n \ge 0$, of the RL  method (\ref{method}) satisfy
$$
\| \bar \rho_n \| \le   C_{c,n} \tau^{p +1} (1+ 2t_n) ||| [u,f,g]^T |||_{p+1,t_n},
$$
where $C_{c,n}=M (C_{el} + \kappa_n  C_{aux}(1+\| K \|))$.
\end{lemma}

\begin{proof}
We will also use the first component of the local error  $\rho_n$, $n \ge 0$, given in (\ref{deflocal}), of the rational discretization (\ref{methodrat}). Subtracting (\ref{deflocalsm}) from (\ref{deflocal}), results in
$$
\bar \rho_n = \rho_n  +\tau (E(\tau) -  E_n(\tau))T_0(t_n)f - (F(\tau) - F_n(\tau))T_1(t_n)g, \quad n \ge 0,
$$
and the proof concludes by using (\ref{localrat}), (\ref{Delta1}) and  (\ref{Delta2}). \end{proof}

The convergence is now a plain consequence of the lemmas. We will assume, by simplicity, $\bar \kappa= \sup_{n \ge 0} \kappa_n < +\infty$. This is obvious when the sequences $\{ \b c_n \}_{n=0}^{\infty}$ and
$\{ \b d_n \}_{n=0}^{\infty}$  are eventually constant. Set  $C_c =M (C_{el} + \kappa  C_{aux}(1+\| K \|))$. General choices of the auxiliar sequences can equally be considered just by using $\bar \kappa_0= \kappa_0$ and $\bar \kappa_n= \max_{0 \le j \le n} \kappa_n$, for $n \ge 1$.

\begin{theorem}
\label{convergence}
Assume that the solution $u$ of (\ref{equdot}) belongs to $ {\cal C}^{(p+1)}[0,+\infty),X)$, that $f \in {\cal C}^{(p+1)}([0,+\infty),X)$ and that $g \in {\cal C}^{(p+2)}([0,+\infty),Y)$. Then, the approximations $\bar u_n$, $ n\ge 1$, generated by the given version of the RL method (\ref{method}) satisfy
$$
\|  u(t_n)- \bar u_n\| \le  C_c \tau^p t_n (1+ 2t_n) ||| [u,f,g]^T |||_{p+1,t_n}, \qquad n \ge 1.
$$
\end{theorem}

\begin{proof}
Since $u(t_n)$, $n \ge 0$,  fits into the recurrence
$$
u(t_{n+1}) = r(\tau A_0)u(t_n) +\tau E_n(\tau) f (t_n+\cdot) - F_n(\tau)g(t_n+\cdot) + \bar \rho_n,
$$
by Lemma~\ref{stability} we have
$$
\|  u(t_n)- \bar u_n\| \le C_s n \max_{0 \le j \le n-1} \| \bar \rho_j \|, \qquad n \ge 1,
$$
and the proof concludes by using Lemma~\ref{consistence}. \end{proof}

Let us point out that, by the well known Lax--Richtmyer Theorem and since the stability of (\ref{method}) depends only on that of $r(\tau A_0)$, it turns out that the RL method (\ref{method}) converges,  without any order of convergence,  for  data in $X \times L^1_{loc}([0,+\infty),X) \times BV_{loc}([0,+\infty),Y)$.

Concerning the IBVP (\ref{problem}), we just  propose
\begin{equation}
\label{ratmetibvp}
\bar w_n = K g(t_n) + \bar u_n, \qquad n \ge 0,
\end{equation}
as the numerical approximation to $w(t_n)$.

Since $u = w-K g$, it turns out that $u \in {\cal C}_{ub}^{p+1}([0,+\infty),X)$ and that
\begin{equation}
\label{triple}
||| [u,f,g]^T |||_{p+1,t_n} \le  ||| [w,f,g]^T |||_{p+1,t_n} + \| K \| \| g \|_{p+1,t_n},
\end{equation}
and then Theorem~\ref{convergence} shows that, for $n \ge 0$,
\begin{equation}
\label{errortw}
\| \bar w_n - w(t_n) \| \le  C_c\tau^p t_n (1+ 2t_n) \left\{ ||| [w,f,g]^T |||_{p+1,t_n} + \| K \| \| g \|_{p+1,t_n} \right\},
\end{equation}
that bounds the error in terms of the smoothness of the solution $w$  and that of  $f$ and  $g$.

\begin{remark}
\label{naturalchoice}
A natural choice for $\b c_n$, $n \ge 0$,  is one fulfilling that the last $p-1$ components of
$ t_n+\tau_n \b c_n $ are the first $p-1$ components of  $t_{n+1} + \tau_{n+1} \b c_{n+1}$ at least for $n \ge n_0$, for a certain $n_0 \le p$. The same argument applies for $\b d_n$ and $p$ of its $p+1$ components.

For instance, with fixed step sizes, we can choose $\b c_0 =[0, 1, \ldots,p-1]$,
$\b c_n = -n + \b c_0$, for $1 \le n \le n_0$ and $\b c_n =  \b c_{n_0}$, for $n > n_0$. For  $\b d_n$, we can proceed
in the same way, just by changing $p-1$ by $p$. In this way, instead of the $p$ evaluations of $f$ and the $p+1$ evaluations of $Kg$ per step, needed for arbitrary choices, we pass to perform just one evaluation of $f$ and $g$, for $n \ge n_0$. If $n_0$ is roughly $p$ or $p+1$, the first evaluations are used for the first  steps, and globally we need one evaluation of $f$ and $g$. Moreover, recalling Remark~\ref{kappa}, a good choice is to deal with sequences  such that, for $n \ge n_0$, $t_n$ is the mean point of both the nodes $t_n + \tau  \b c_n$ and $t_n + \tau \b d_n$. This requires, to use half-integer numbers and different evaluation points for $f$ and for  $Kg$, and both are harmless features.

If we do not insist on using these natural choices, we will pay the price of performing more  evaluations, but we are free to choose the auxiliar nodes and improve the value of $\kappa$.
\end{remark}

\begin{remark}
Suppose that  $r(z)$ is the stability function of an  $s$-stage Runge--Kutta method and that $g'$ is really available or computable with high precision. Besides the central point of the order reduction phenomenon, overcome by the RL method, both the  Runge--Kutta method  and method (\ref{method}), applied to (\ref{equdot}), require solving $s$ linear systems with the operator $A_0$ per step. But while the  Runge Kutta method needs $s$ evaluations of $f$ and $Kg'$  per step,  with an adequate choice  the auxiliar nodes, as explained in Remark~{\ref{naturalchoice}},  the RL method (\ref{method}) requires just one new evaluation of $f$ and $Kg$ per step (against the $p$ and $p+1$ required with general choices of the nodes). Notice that the vector of abcissa $\b c_{RK}$ of the Runge--Kutta method is neither used in (\ref{method}) nor related to $\b c_n's$ or $\b d_n's$.
\end{remark}

\begin{remark}
The approach we suggest can be easily extended to deal with higher order derivatives in the source terms, that  is, to consider time discretizations of problems with the format
$$
u'(t) = A_0 u(t) + \sum_{j=0}^m K_j f_j^{(j)}(t), \qquad t \ge 0,
$$
where $K_j : Y_j \to X$ are linear, bounded operators, acting on different Banach spaces $Y_j$, $0 \le j \le m$, and the nonhomogeneous terms are mappings $f_j \in {\cal C}_{ub}^{j}([0,+\infty),Y_j)$, $0 \le j \le m$. In particular, $m=2$ would be of interest in the context of wave equations.
\end{remark}

\begin{remark}
\label{variabless}
Given a sequence of step-sizes  $\tau_n$, $n \ge 0$, fulfilling the requirements in Remark~\ref{remarkrmpv}, in particular the constant ones,  the suggested method (\ref{method}) becomes
$$
\bar u_{n+1} = r(\tau_n A_0) \bar u_n + \tau_n E_n(\tau_n) f(t_n + \cdot) - F_n(\tau_n)g(t_n + \cdot), \qquad n \ge 0.
$$
and in Theorem~\ref{convergence} we must change $t_n \tau^p$ by $\sum_{j=0}^{n-1} \tau_j^{p+1}$. The main interest of variable step sizes is to adapt the step size so as to take advantage of the local regularity, something that, in view of Remark~\ref{remarkrmpv},  looks possible in some cases.
\end{remark}

\section{Full discretization}
\label{SFD}
In this section we consider  the full discretization  of (\ref{problem}). The framework we introduce does not assume any specific nature of the space discretization, so that we use the term space discretization in general, but in the comments we refer to  standard finite differences, finite elements methods or spectral methods. We adopt the  point of view of the method of lines,

Let $ 0 < h \le h_0$ denote the parameter governing the space discretization. All the operators that we consider are going to be linear. Associated with each $h$-value,  we introduce two Banach spaces  $X_h$ and $Y_h$ of finite dimension and two operators $P_h : X \to X_h$ and $Q_h : Y \to Y_h$.  Their norms, as well as the  norms of associated  operators are, by default, denoted by $\| \cdot \|_h$. The norms $\| \cdot \|_h$ reflect the ones of  $X_h$ and $Y_h$ and it  is overunderstood that, for $w \in X$,  $P_hw \in X_h$ contains enough information so as to provide an approximation of $w \in W$.  The same idea is behind $Q_hv$, for $v \in Y$.
For instance, in finite elements, $P_h$ and $Q_h$ are $L^2$ projectors (considered even for $X=L^q$, $q \ne 2$) and, in finite differences, they are  sampling operators (or concentrated averages, for $q < +\infty$).

\medskip
The connection to IBVP (\ref{problem}) is given by:

\begin{itemize}

\item[] (a) An onto  operator $\partial_h : X_h \to Y_h$.   We set $X_{h,0} = \mbox{Ker}\,(\partial_h)$.

\item[]  (b) Two   operators $A_h : X_h \to X_h$ and $A_{h,0}:  X_{h,0} \to X_{h,0}$ such that $A_h = A_{h,0}$ on $X_{h,0}$.  The semigroup generated by $A_{h,0}$ on $X_{h,0}$  is denoted by $S_{A_{h,0}}(t)$, $ t \ge 0$.

\item[] (c) An operator $K_h : Y_h \to X_h$.

\item[] (d) An operator $P_{h,0} : X \to X_{h,0}$ such that $\|P_h-P_{h,0}\|_{W_0 \to X_h} \le C' \|\partial_h  P_h\|_{W_0 \to Y_h}$ for some constant $C'$ and another operator $K_{h,0}:Y_h \to X_{h,0}$. If $X_h \subset X$ so that $P_{h,0}$ is also defined in $X_h$, a usual condition for finite elements, then one can substitute $K_{h,0}$ by $K_{h,0} Q_h g=P_{h,0}(K_h Q_h g)$ and it is not necessary to include a new operator.

\end{itemize}
As the notation suggests, $A_h$, $A_{h,0}$, $\partial_h P_h$ and $K_h Q_h$, try to fit with $P_h A$, $P_{h,0} A_0$, $Q_h\partial$, $P_hK $.

For finite elements, there is an obvious, natural connection between $P_h$ and $P_{h,0}$, but nothing is postulated about this point. The idea is that $P_{h,0} w$ contains specific information to reconstruct $w \in\mbox{\rm Ker} ( \partial)$.  In spectral methods, for boundary conditions other than the Dirichlet ones, it turns out that $\partial w =0$  does not imply that $\partial_h P_h w=0$, for $w \in X_h$, this is why it is important to introduce $P_{h,0}$.
\medskip

The consistency refers to a couple of Banach spaces $(W,\| \cdot \|_Z)$, with $ W \subset D(A) \subset X$ and $(Z, \| \cdot\|_Z)$, with $Z \subset Y$ that are continuously embedded in  $D(A)$ and $Y$  and satisfy (\ref{WZ}).
We set $ W_0 = \mbox{\rm Ker} \, (\delta) \cap W$, endowed with the norm $ \| \cdot \|_{W_0}$, that is the one induced by $\| \cdot \|_W$.

To reflect the fact that the norm of $W$ and $Z$ are finer than those of $X$ and $Y$, we must consider new, adequate norms $\| \cdot \|_{h,W}$ on $X_h$ and $\| \cdot \|_{h,Z}$  on $Y_h$. The spaces $(W_h, \| \cdot \|_{h,Z} )$ and  $(Z_h, \| \cdot \|_{Z_h})$ are denoted by $W_h$ and $Z_h$,  i.e., $W_h = X_h$ and $Z_h =Y_h$,  but endowed with different norms.

Set
$$L_h = \max\{\| P_h \|_h, \| P_{h,0}\|_h, \| Q_h \|_h, \| \partial_h \|_{W_h \to Y_h}, \| Q_h \|_{Z \to Z_h} \},$$
and
$$
M_h = \sup_{ t \ge 0}\| S_{A_{h,0}}(t) \|.
$$
The ideal stability hypotheses are

\begin{itemize}
\item[] (S.1) $L := \sup_{0 < h \le h_0} L_h <+\infty$.
\item[] (S.2) $M_s: =  \sup_{0 < h \le h_0} M_h <+\infty$.
\end{itemize}

However, frequently $L_h$ and, particularly, $M_h$ exhibit a weak growth, for instance $M_h = O(|\ln h|)$, as $ h \to 0+$  (this is why we have singled out $M_h$). By simplicity, we will assume (S.1) and (S.2), keeping in mind that the final convergence estimate is valid by using $L_h$ and $M_h$.

Consistency is  expressed in terms of an infinitesimum $\epsilon: (0,h_0] \to (0,+\infty)$  that measures the quality of the convergence. Typically $\epsilon (h) = Const.\, h^m$, for some order $m > 0$. Tracing the value of all the involved constants, by following their proofs,  is an impossible task that leads to pessimistic and huge values for a general situation. In the absence of particular properties, we just try to catch the order and later, if possible, to use other techniques to get the leading constant for a particular problem. With this idea in mind, we will simplify the statements and proofs by expressing the different infinitesima in the form $O(\epsilon(h))$, overunderstanding that the hidden constants are bounded by a common value, uniformly for $0 < h \le h_0$.

Consistency  assumes that for $0 < h \le h_0$:

\begin{itemize}
\item[] (C.1) $\| A_{h,0}P_{h,0}- P_{h,0} A_0 \|_{W_0 \to X_h}  \le O(\epsilon(h))$.
\item[] (C.2) $\| \partial_h P_h \|_{W_0 \to  Y_h} \le O(\epsilon(h))$.
\item[] (C.3) $\| P_h K(0)  - K_h  Q_h \|_{Z \to X_h}   \le O(\epsilon(h) )$.
\item[] (C.4) $\| P_{h,0} K(0)  - K_{h,0}  Q_h] \|_{Z \to X_{h,0}}   \le O(\epsilon(h) )$.
\end{itemize}

It is important to note that, while (C.1) is standard for IVPs,  when dealing with IBVPs,  it is natural to require that both the  boundary condition and problem (\ref{eig}) can be discretized accurately for data in the consistency class, that is (C.3) and (C.4), and that the solution depends continuously on the data, that is the reason of (C.2).

Let $w$ be the solution of the IBVP (\ref{problem}). Henceforth we suppose that $w \in {\cal C}_{ub}^1([0,+\infty),W)$,  $f \in {\cal C}_{ub}([0,+\infty),X)$ and $g  \in {\cal C}_{ub}^1([0,+\infty),Z)$ and set $w=u+K(0) g$, where $u$ solves (\ref{equdotu}). We have already remarked  that $u$, due to the natural compatility condition,  not only takes values in $W_0$
but is also the solution of (\ref{equdotu}), that turns out to be an IVP in $W_0$. We remark here that, when using finite differences, this space $W$ consists of continuous functions, for which the nodal projection $P_{h,0}$ is completely defined and there is no need to recur to concentrated averages.

Given $ t \ge 0$, while for   $K(0) g(t)$ the goal is to approximate $P_h K(0) g(t)$, that brings the information to reconstruct $K(0) g(t)$, for $u(t) \in W_0$ it is rather to approximate $P_{h,0} u(t)$, that codifies the information to reconstruct $u(t)$. With these goals in mind, we just propose $K_h  Q_h g'(t)$ as the space discretization of $K(0) g'(t)$, and  $u_h(t)$, where $u_h$ is the solution  of the IVP
\begin{equation}
\label{ivpuh}
\left\{
\begin{array}{lcl}
u_h'(t)  &=& A_{h,0} u_h(t) + P_{h,0} f(t) - K_{h,0} Q_h g'(t), \quad t \ge 0, \\
u_h(0) &= &  P_{h,0}w_0 - K_{h,0} Q_h g(0), \\
\partial_h u_h(t) &=& 0, \quad t \ge 0,
\end{array} \right.
\end{equation}
as the space discretization of $u$.  The proposed space discretization of $w$ is simply $w_h = K_h Q_h g + u_h$, due to the consistency hypothesis (C.3).

We stress that (\ref{ivpuh}) is an IVP in $X_{h,0}$ indeed (the boundary condition is  thus redundant). In contrast, notice  that  the related problem with source term $P_h f(t) - K_h Q_h g'(t)$ results in an IVP in $X_h$, not in $X_{h,0}$, no matter that its solution takes values in $X_{h,0}$. Actually, the error analysis of  such an IVP looks impossible, unless the extra hypothesis
$\| P_h - P_{h,0} \|_{W_0 \to X_h}=O(\epsilon(h))$, which comes from (d) and (C.2), is introduced in the formalism.

We are now in a position to prove the next result. In consonance with (\ref{defnorm}), we set
$$
 \| w \|_{1,t,W} =  \max_{0 \le j \le 1} \max_{0 \le s \le t} \| w^{(j)} (s) \|_W, \quad t \ge 0.
$$

\begin{theorem}
\label{convergencesd}
Assume that the solution $w$ of the IBVP (\ref{problem}) belongs to  ${\cal C}_{ub}^1([0,+\infty),W)$,  $f \in {\cal C}_{ub}([0,+\infty),X)$ and $g  \in {\cal C}_{ub}^1([0,+\infty),Z)$. Then, the error of the full semidiscrete approximation $e_h=P_{h,0}u-u_h$ can be estimated by
$$
\| e_h  \|_h \le t M_s  O( \epsilon(h) ) \| w \|_{1,t,W},\qquad  t \ge 0, \quad 0 < h \le h_0. \,
$$
\end{theorem}
\begin{proof}
Plugging $P_{h,0}$ in (\ref{equdotu}) and making the difference with (\ref{ivpuh})  shows readily that the error $e_h = P_{h,0} u - u_h$ fits into the IVP
$$
\left\{
\begin{array}{lcl}
e_h'(t)  &=& A_{h,0} e_h(t) + \delta_h(t), \quad t \ge 0, \\
e_h(0) &= & (K_{h,0} Q_h-P_{h,0}K(0))g(0), \\
\partial_h e_h(t) &=& 0, \quad t \ge 0,
\end{array} \right.
$$
where the truncation error is
$$
\delta_h(t) =P_{h,0} A_0 u(t) - A_{h,0} P_{h,0} u(t)-(P_{h,0}K(0)g'(t)-K_{h,0}Q_h g'(t)), \qquad t \ge 0.
$$
Therefore (recall (\ref{defnorm})),  by (C.1), (C.4) and the fundamental estimate (\ref{cotaW0u}), for $t \ge 0$,
\begin{eqnarray}
\| \delta_h(t) \|_h \le O(\epsilon(h) ( \| u \|_{0,t,W} + \| g \|_{1,t,Z}) = O(\epsilon(h)) \| w \|_{1,t,W}.
\nonumber
\end{eqnarray}
whereas
$$
\|e_h(0)\|_h \le O(\epsilon(h)) \|g(0)\|_Z.$$
Now, by taking norms in the variation-of-constants formula,
$$
e_h(t) =  S_{A_{h,0}}(t) e_h(0)+\int_0^t S_{A_{h,0}}(t-s) \delta_h(s) \,\mbox{\rm d}s, \qquad t \ge 0,
$$
and using (S.2),
we readily conclude the proof.
%\begin{equation}
%\label{erroruh}
%\| e_h (t) \|_h = M_s  t O(\epsilon(h)) \| w \|_{1,t,W}, \qquad t \ge 0.
%\end{equation}
%\textcolor{red}{Considering now that
%\begin{eqnarray}
%P_h w-w_h&=& P_h(u+K(0)g)-(u_h+K_h Q_h g) \nonumber \\
%&=&(P_h-P_{h,0})u+(P_{h,0}u-u_h)+(P_h K(0)g-K_h Q_h g), \nonumber
%\end{eqnarray}}
\end{proof}

For $0 < h \le h_0$ and $t_n = n\tau$, $n \ge 0$, it is now natural to propose, as the full approximation to $w(t_n)$,   the sum
$$
\bar w_{h,n} = K_h Q_h g(t_n) + \bar u_{h,n},
$$
where $\bar u_{h,n}$ is the time approximation, at time level $n$,  to   the solution $u_h$ of the IVP
\begin{equation}
\label{equdoth}
\left\{
\begin{array}{lcl}
u'_h(t) &= &A_{h,0} u_h(t) + P_{h,0} f(t) - K_{h,0} Q_h g'(t), \qquad t \ge 0, \\
u_h(0) &=& u_{h,0}:=P_h w(0)- K_{h,0} Q_h g(0),
\end{array} \right.
\end{equation}
provided by some version of a RL method (\ref{method}),  based on $r(z)$. Since
\begin{eqnarray}
\| P_h w(t_n) - \bar w_{h,n} \|_h &\le & \|(P_h -P_{h,0})u(t_n)\|_h+\| P_{h,0} u(t_n) -u_h(t_n) \|_h
 \nonumber \\
 &&+ \| u_h(t_n) - \bar u_{h,n} \|_h+\|[P_h K(0)-K_h Q_h]g(t_n)\|_h,
\nonumber
\end{eqnarray}
and combining Theorems \ref{convergence} and \ref{defnorm}, (d), (C.2) and (C.3), we get the estimate for the error of the full discretization that we just  state afterwards. Recalling Theorem \ref{convergence},  notice that the estimate for the time discretization of (\ref{ivpuh}) uses the norm
$$
||| [u_{h,0} ,P_h f,Q_h  g]^T |||_{p+1,t,X_h} = \| u_{h,0}  \|_h + \max_{0 \le j \le p+1} \| P_h f^{(j)}(s) \|_h +  \max_{0 \le j \le p+2} \| Q_h g^{(j)}(s) \|_h,
$$
that, by (S.1), is bounded by $M_L  |||[u,f,g]|||_{p+1,t}$, for $t \ge 0$.

\begin{theorem}
Assume that the solution $w$ of the IBVP (\ref{problem}) belongs to  ${\cal C}_{ub}^1([0,+\infty),W) \cap {\cal C}_{ub}^{p+1}([0,+\infty),X) $,  that $f \in {\cal C}_{ub}^{p+1}([0,+\infty),X)$ and  that $g  \in {\cal C}_{ub}^{p+2}([0,+\infty),Z)$. Let $0 < h \le h_0$, $\tau >0$ and $n \ge 0$. Then, the error of the full discretization can be estimated by
$$
\| \bar w_{h,n} - P_h w(t_n) \|_h \le ERT_n(\tau) + ERS_n(h),
$$
where
$$ ERT_n(\tau) =   C_c\tau^p t_n (1+ 2t_n) \left\{ ||| [w,f,g]^T |||_{p+1,t_n} + \| K(0) \| \| g \|_{p+2,t_n} \right\}, $$
and
$$
ERS_n(h) = t_n M_s  O( \epsilon(h) ) \| w \|_{1,t,W}.
$$
\label{conv_final}
\end{theorem}
Again by  Lax--Milgram Theorem, it can be proved that the proposed full discretization of (\ref{problem}) converges  for data $w_0 \in X$, $f \in {\cal C}_{ub}([0,+\infty),X)$ and $g : [0,+\infty) \to Z$ of finite total variation.

\section{Numerical experiments}

The aim of this section is to corroborate the  previous results. We consider the domain $\Omega = (0,1)\times(0,1) \subset \mathbb{R}^2$ and we integrate the following parabolic initial boundary value problem with Dirichlet boundary conditions
\begin{eqnarray}
\left\{ \begin{array}{rcl} w_t(t,x,y) & = & \Delta w(t,x,y) - \sin(x+y+t)  + 2\cos(x+y+t),  \quad (x,y) \in \Omega,\, t \in (0,T), \\
w(0,x,y)  & = & \cos(x+y),  \quad (x,y) \in \Omega,  \\
w(t,x,y) & = & \cos(x+y+t), \quad (x,y) \in \partial \Omega,\, t \in (0,T), \end{array} \right.
\label{problem2}
\end{eqnarray}
for some $T>0$.
This problem has the form of (\ref{problem}) and satisfies (\ref{necreg0D}), so $w$ is a genuine solution of (\ref{problem}) which in fact has as high regularity as required for any of the theorems of the previous section. We notice that
 $X=L^2(\Omega)$, $A=\Delta$, $D(A)=H^2(\Omega)$, $Y=H^{\frac{3}{2}}(\partial \Omega)$  and $A_0=A|_{\mbox{ker}(\partial)}$ is the infinitesimal generator of an analytic semigroup of negative type, such that $D(A_0)=H^2(\Omega)\cap H_0^1(\Omega)$ \cite{pazy}. We also notice that $f$ and $g$ have as much regularity as required in Theorems \ref{convergence} and \ref{conv_final}.

For the space discretization, we use the well-known $4$th-order nine-point formula for the Laplacian \cite{S}, which gives rise to $K_h$ when considering all grid nodes and $K_{h,0}$ just for the interior ones. In such a case, $P_h$ is the projection on the nodes of the grid when applied to a $C^1$-function and, for any function in $L^2(\Omega)$, as $C^1(\Omega)$ is dense in $L^2(\Omega)$, it corresponds to the limit in the discrete $L^2$-norm. The same applies for $Q_h$ and the projection on the nodes of the boundary of the grid. Then, $P_{h,0} f$ is not only basically the projection on the interior points of the grid, but it also contains some information on the nodes of the boundary, as it comes from the nine-point formula. All these operators are thus bounded in the discrete $L^2$-norm and then (S.1) holds. On the other hand, $A_{h,0}$ can be seen to have negative eigenvalues for every $h>0$, so that (S.2) holds in the same norm.

Finally, for $W=H^6(\Omega)$ and $Z=H^{\frac{11}{2}}(\Omega)$, (C.1),(C.3) and (C.4) are well-known to hold with $\varepsilon_h=O(h^4)$ and (C.2) with $\varepsilon_h=0$.

\subsection{SDIRK}

Problem (\ref{problem2}) is first integrated in time by the $3$-stages SDIRK Runge--Kutta method of classical order $p=4$ and stage order $q=1$ \cite{HW} by using the standard method of lines, i.e. discretizing first in space the corresponding problem (\ref{problem}) and then in time. In such a way, three linear systems with sparse matrices  must be solved per step, as well as three evaluations of $f$, the same for $g$ and the same for $g'$. We have integrated till time $T=1/2$ with $N=200$ so that the error in space is negligible and the linear systems have been solved with the iterative gradient conjugate method with $tol=10^{-14}$. The results on the global error turn up in the second column of  Table \ref{exp1SDIRK}. We notice that an order greater than that expected by its stage order \cite{AP2,OR} is obtained (notice that $u_0\in D(A_0^{1/4})$ in this case \cite{pazy}). That is due to a summation-by-parts argument, which in parabolic problems explains that the global order behaves as the local one, instead of one less, when $r_\infty \neq 1$ \cite{HW, OR}. In fact, the local order can be seen to also behave as $O(\tau^{2+1/4})$, as Table \ref{exp1loc} shows.

\begin{table}[htpb]
\centering
\caption{Global errors and order of convergence with SDIRK3 method and suggested
 rational SDIRK3 methods.}
\begin{tabular}{rrrrrrrrr}
%\toprule%
& \multicolumn{2}{c}{Runge--Kutta} && \multicolumn{2}{c}{Rational Explicit} && \multicolumn{2}{c}{Rational Implicit}  \\
%\cline{2-3} \cline{5-6}  \cline{8-9}
%\cmidrule{2-3} \cmidrule{5-6} \cmidrule{8-9}
\multicolumn{1}{l}{ step size} & \multicolumn{1}{l}{error} & \multicolumn{1}{l}{order} & &
\multicolumn{1}{l}{error\rule{0pt}{2.5ex}\rule{0pt}{2.5ex}} & \multicolumn{1}{l}{order} & &
\multicolumn{1}{l}{error\rule{0pt}{2.5ex}\rule{0pt}{2.5ex}} & \multicolumn{1}{l}{order}\\
\hline

1.000e-01  & 2.379e-04  & --          & & 8.721e-07  & -- & & 1.797e-07 & --\\

5.000e-02  & 5.083e-05  & 2.23       & & 6.656e-08  & 3.71 & & 8.792e-09 & 4.35 \\

2.500e-02  & 1.083e-05  & 2.23       & & 4.025e-09  & 3.98 & & 4.874e-10 & 4.17\\

1.250e-02  & 2.254e-06  & 2.26       & & 2.555e-10  & 4.04 & &  3.386e-11 & 3.85\\

6.250e-03  & 4.672e-07  & 2.27       & & 1.569e-11  & 4.03 & & 2.667e-12 & 3.67\\
\label{exp1SDIRK}
\end{tabular}
\end{table}

\begin{table}[htpb]
\centering
\caption{Local errors and local order of convergence with SDIRK3 method and suggested rational SDIRK3 methods.}
\begin{tabular}{rrrrrrrrr}
%\toprule%
& \multicolumn{2}{c}{Runge--Kutta} && \multicolumn{2}{c}{Rational Explicit} && \multicolumn{2}{c}{Rational Implicit}  \\
%\cline{2-3} \cline{5-6}  \cline{8-9}
%\cmidrule{2-3} \cmidrule{5-6} \cmidrule{8-9} %
\multicolumn{1}{l}{ step size} & \multicolumn{1}{l}{error} & \multicolumn{1}{l}{order} & &
\multicolumn{1}{l}{error\rule{0pt}{2.5ex}\rule{0pt}{2.5ex}} & \multicolumn{1}{l}{order} & &
\multicolumn{1}{l}{error\rule{0pt}{2.5ex}\rule{0pt}{2.5ex}} & \multicolumn{1}{l}{order}\\
\hline

%2.000e-01  & 1.822e-03  & --       & & 3.415e-05  & -- & & 7.593e-07 & -- \\

1.000e-01  & 3.964e-04  & --      & & 1.308e-06  & -- & & 4.727e-08 & -- \\

5.000e-02  & 8.315e-05  & 2.25       & & 5.195e-08  & 4.65 & & 2.223e-09 & 4.41 \\

2.500e-02  & 1.734e-05  & 2.26       & & 2.024e-09  & 4.68 & &  8.524e-11 & 4.70 \\

1.667e-02  & 6.921e-06  & 2.27       & & 2.981e-10  & 4.72 & & 1.239e-11 & 4.76 \\

\label{exp1loc}
\end{tabular}
\end{table}

On the other hand, we have integrated the same problem with the suggested technique in this paper. We have considered two possibilities for the nodes $\mathbf{c}$ and $\mathbf{d}$ in Section 5, which we will denote by explicit and implicit method because of the difference in $\mathbf{c}^n$ for $n\ge 3$:
\begin{eqnarray}
\left\{\begin{array}{rclrcl}
\mathbf{c^0}&=&[0, 1, 2, 3],  &\mathbf{d^0}&=&[0, 1, 2, 3, 4],  \\
 \mathbf{c^1}&=&[-1, 0, 1, 2], & \mathbf{d^1}&=&[-1, 0, 1, 2, 3],  \\
\mathbf{c^2}&=&[-2, -1, 0, 1], & \mathbf{d^2}&=&[-2, -1, 0, 1, 2],  \\
\mathbf{c^n}&=&[-3, -2, -1, 0], & \mathbf{d^n}&=&[-3, -2, -1, 0, 1], \quad \mbox{  explicit method, }  n\ge 3 \\
\mathbf{c^n}&=&[-2, -1, 0, 1], & \mathbf{d^n}&=&[-3, -2, -1, 0, 1], \quad \mbox{ implicit method.} \end{array} \right.\label{cd}
\end{eqnarray}
(Notice that, as $p=4$, $\mathbf{c}^n \in \mathbb{R}^4$ while  $\mathbf{d}^n \in \mathbb{R}^5$. There are many other possibilities but we have chosen these among the ones which just imply at most one function evaluation of $f$ and $g$ per step.) Moreover, as in this case, the matrix in Butcher tableau is $\bar{W}=\delta I+N$ for a certain value $\delta>0$ and a nilpotent matrix $N$, it happens that
\begin{eqnarray}
r(z)&=&1+z \mathbf{b}^T (I-z \bar{W})^{-1} \mathbf{e}=1+z \mathbf{b}^T [(I-z \delta)I -zN]^{-1} \mathbf{e} \nonumber \\
&=&1+\frac{z}{1-z \delta}\mathbf{b}^T [I-\frac{z}{1-z \delta}N]^{-1}\mathbf{e} \nonumber \\
&=& (1-\frac{1}{\delta})+r_{11}\frac{1}{1-z \delta}+r_{12}\frac{1}{(1-z \delta)^2}+r_{13}\frac{1}{(1-z \delta)^3}, \nonumber
\end{eqnarray}
for some values $r_{11}, r_{12}, r_{13}$, taking into account that
$$
\frac{z}{1-z \delta}=-\frac{1}{\delta}+\frac{1}{\delta}\frac{1}{1-z  \delta},$$
and that the method is consistent. In such a way, $k=1$, $m_1=3$ and $w_1=\delta$ in (\ref{rz}). Then,
Lemma 4.3 in \cite{ArP} can be applied to
\begin{eqnarray}
H_{11}(z)&=&\frac{1}{1-z \delta}=1+\delta z +\delta^2 z^2+\delta^3 z^3 +O(z^4), \nonumber \\
H_{12}(z)&=&\frac{1}{(1-z \delta)^2}=1+2\delta z +3\delta^2 z^2+4\delta^3 z^3 +O(z^4), \nonumber \\
H_{13}(z)&=&\frac{1}{(1-z \delta)^3}=1+3\delta z +6\delta^2 z^2+10\delta^3 z^3 +O(z^4), \nonumber \\
I_{11}(z)&=&\frac{z}{1-z \delta}=z+\delta z^2 +\delta^2 z^3+\delta^3 z^4 +O(z^5), \nonumber \\
I_{12}(z)&=&\frac{z}{(1-z \delta)^2}=z+2\delta z^2 +3\delta^2 z^3+4\delta^3 z^4 +O(z^5), \nonumber \\
I_{13}(z)&=&\frac{z}{(1-z \delta)^3}=z+3\delta z^2 +6\delta^2 z^3+10\delta^3 z^4 +O(z^5), \nonumber
\end{eqnarray}
giving rise to the coefficients $\mathbf{\gamma}_{n,1,i} \in \mathbb{R}^4$ and  $\mathbf{\eta}_{n,1,i} \in \mathbb{R}^5$ for $i=1,2,3$, which turn up in (\ref{En}) and (\ref{Fn}).

Then, at each step four linear systems with sparse matrices must be solved and one evaluation of $f$ and $g$ are required. (We notice that the additional linear system to be solved with respect to the Runge--Kutta case comes from the calculation of $K_{h,0}$ in the part corresponding to (\ref{Fn}) when integrating (\ref{equdoth})). Table \ref{exp1SDIRK} shows the global error for several stepsizes for both the explicit and implicit rational implementation suggested in this paper. It can be observed that, in both cases, the order is very near $4$, which is the classical order of the method, as Theorem \ref{conv_final} predicts. On the other hand, the local order corresponding to the error at the first time stepsize when using $\mathbf{c}$ and $\mathbf{d}$ in the last lines of (\ref{cd}) approaches $5$ as $\tau$ diminishes, as Table \ref{exp1loc} shows and as predicted by Lemma \ref{consistence}.

\begin{figure*}
%\vspace{-6cm}
\centerline{
\includegraphics[width=90mm]{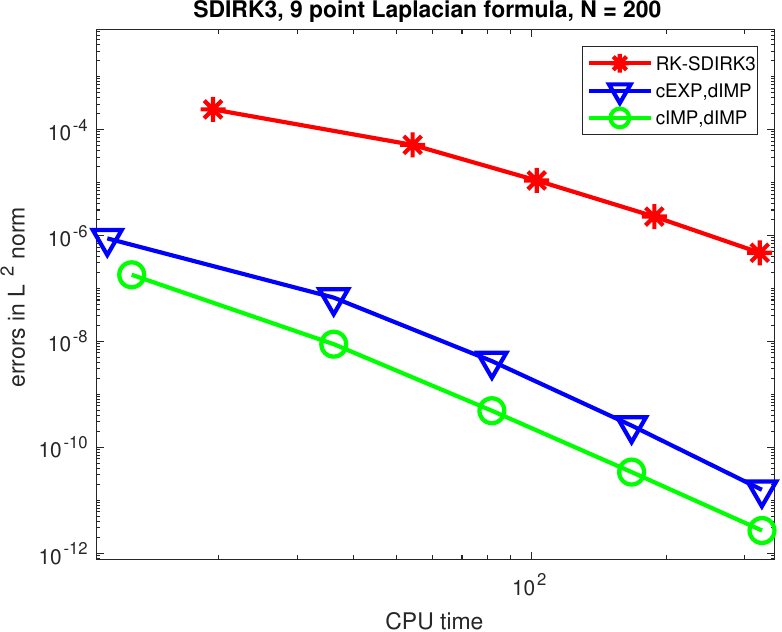}
}
%\vspace{-6cm}
\caption{Error against CPU time when integrating problem (\ref{problem2}) with SDIRK3 method and suggested rational SDIRK3 methods.} \label{f1}
\end{figure*}

Comparing the size of the errors for a fixed stepsize, we can see that, even for the biggest time stepsize $\tau=0.1$, the errors are much smaller with the rational suggestion in this paper than using the Runge--Kutta method through the standard method of lines. As order reduction is also avoided, that difference is even much more significant when $\tau$ diminishes. Moreover, at least in this problem, the size of the error for a same stepsize is quite smaller with the implicit choice of $\mathbf{c}$ than with the explicit one, which is reasonable because of Remark \ref{kappa}. We also show here a comparison in terms of CPU time among the three methods with our implementations in MATLAB. In particular, Figure \ref{f1} shows the global error against CPU time and we can observe that, for the biggest time stepsize, the CPU time which is required with the Runge--Kutta method is bigger than that corresponding to the rational suggestion here. It thus seems that, at least in such a case, the extra evaluations of $f,g,\dot{g}$ and the linear combinations to be done with the Runge-Kutta method take more time than the resolution of an extra linear system in the rational methods. In any case, for this problem and a fixed cpu time, the difference in accuracy is of four orders of accuracy for the bigger time-stepsizes between the Runge--Kutta method and the implicit rational method and of five orders of accuracy for the smallest stepsizes. And what is also very important, $\dot{g}$ has not been required with the rational suggestion.

\subsection{Gauss}

We have also taken as time integrator the three-stages Gauss Runge--Kutta method, which is well-known to have classical order $p=6$ and stage order $q=3$ \cite{HW}. When using it to integrate (\ref{problem2}) till time $T=1$, order $3.25$ for the global error would be expected in general according to \cite{AP2}. However, in the same way as for DIRK method, a summation by parts argument applies because the problem is parabolic and the global order is one order higher, i.e. 4.25, as it can be approximately observed in the third column of Table \ref{exp1Gauss}. We notice that in this case, since the matrix in Butcher tableau diagonalizes and $r(-\infty)=-1$, $r(z)$ can be written as
$$r(z)=-1+\frac{r_{11}}{1-w_1 z}+  \frac{r_{21}}{1-w_2 z}+\frac{r_{31}}{1-w_2 z},$$
with $r_{l,1}\in \mathbb{R}$ and $w_l>0$ for $l=1,2,3$. Using this, the implementation of the Runge--Kutta method means solving,  at each step,
three linear systems with sparse matrices as well as the three evaluations of $f(t_n+c_l \tau)$,  $g(t_n+c_l \tau)$ and $\dot{g}(t_n+c_l \tau)$ for $l=1,2,3$.

\begin{table}[htpb]
\centering
\caption{Errors and order of convergence with  Gauss3 method and the suggested rational Gauss3 methods.}
\begin{tabular}{lllllllllll}
%\toprule%
& \multicolumn{2}{c}{Runge--Kutta} & \multicolumn{2}{c}{Explicit Rational} & \multicolumn{2}{c}{Implicit Rational } & \multicolumn{2}{c}{Centered Rational} & \multicolumn{2}{c}{Chebyshev-Centered Rat. } \\
%\cline{2-3} \cline{5-6}  \cline{8-9}
%\cmidrule{2-3} \cmidrule{5-6} \cmidrule{8-9}
\multicolumn{1}{l}{ step size} & \multicolumn{1}{l}{error} & \multicolumn{1}{l}{order} & &
\multicolumn{1}{l}{error\rule{0pt}{2.5ex}\rule{0pt}{2.5ex}} & \multicolumn{1}{l}{order} & &
\multicolumn{1}{l}{error\rule{0pt}{2.5ex}\rule{0pt}{2.5ex}} & \multicolumn{1}{l}{order}\\
\hline

1.000e-01  & 6.117e-06  & --       &  2.276e-08  & -- &  1.712e-09 & -- & 7.636e-11 & -- & 1.027e-10 & --\\

6.667e-02  & 1.105e-06  & 4.22       &  2.029e-09  & 5.96 &  1.291e-10 & 6.38 & 5.593e-12 & 6.45 &  6.369e-12 & 6.86 \\

5.000e-02  & 3.131e-07  & 4.38       &  3.576e-10  & 6.03 &   2.100e-11 & 6.31 &  8.259e-13& 6.65&  9.140e-13& 6.75\\

4.000e-02  & 1.225e-07  & 4.20       &  9.245e-11  & 6.06 &  5.408e-12 & 6.08 &  1.654e-13& 7.21&  4.390e-13 & *\\

3.333e-02  & 5.508e-08  & 4.39       &  3.059e-11  & 6.07 &  1.819e-12 & 5.98 &  1.737e-13 & *&  2.508e-13 & *\\
\label{exp1Gauss}
\end{tabular}
\end{table}

\begin{figure*}
%\vspace{-6cm}
\centerline{\includegraphics[width=90mm]{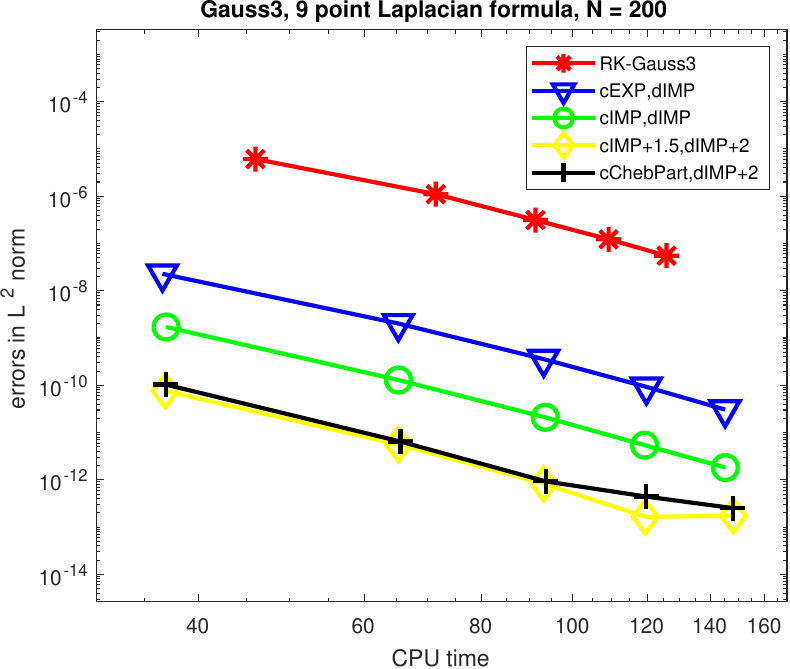}}
%\vspace{-6cm}
\caption{Error against CPU time when integrating problem (\ref{problem2}) with GAUSS3 method and suggested rational GAUSS3 methods.} \label{f2}

\end{figure*}

The global errors which are obtained with the rational methods can be observed in Table \ref{exp1Gauss}, where now, more diversity of $\mathbf{c}$ and $\mathbf{d}$ has been taken, considering Remark \ref{kappa}. More precisely, the equivalent choice to (\ref{cd}) with
\begin{eqnarray}
\left\{\begin{array}{rclrcl}
%\mathbf{c^0}&=&[0, 1, 2, 3, 4 ,5]  &\mathbf{d^0}&=&[0, 1, 2, 3, 4, 5, 6]  \\
% \mathbf{c^1}&=&[-1, 0, 1, 2, 3, 4] & \mathbf{d^1}&=&[-1, 0, 1, 2, 3, 4, 5]  \\
%\mathbf{c^2}&=&[-2, -1, 0, 1, 2, 3] & \mathbf{d^2}&=&[-2, -1, 0, 1, 2, 3, 4]  \\
%\mathbf{c^3}&=&[-3,-2, -1, 0, 1, 2] & \mathbf{d^3}&=&[-3, -2, -1, 0, 1, 2, 3]  \\
%\mathbf{c^4}&=&[-4,-3,-2, -1, 0, 1] & \mathbf{d^4}&=&[-4,-3, -2, -1, 0, 1, 2]  \\
\mathbf{c^n}&=&[-5,-4, -3, -2, -1, 0], & \mathbf{d^n}&=&[-5,-4,-3, -2, -1, 0, 1], \quad \mbox{  explicit method, }  n\ge 5, \\
\mathbf{c^n}&=&[-4,-3,-2, -1, 0, 1], & \mathbf{d^n}&=&[-5,-4,-3, -2, -1, 0, 1], \quad \mbox{ implicit method, }
\end{array} \right.\nonumber
\end{eqnarray}
and the centered integer and half-integer nodes
\begin{eqnarray}
\mathbf{c^n}=[-2.5,-1.5,-0.5, 0.5, 1.5, 2.5], \quad \mathbf{d^n}=[-3, -2, -1, 0, 1,2,3], \quad \mbox{ centered method, } n\ge 3.\label{centered}
\end{eqnarray}
With these choices, four linear systems with sparse matrices  must be solved at each step and, by keeping the calculation from one step to the other, just one function evaluation of $f$ and $g$.

Finally, we have also considered the method for which $\mathbf{c}$ corresponds to the roots of  Chebyshev polynomial $T_3(x)$ relocated in the intervals $[-1,0]$ and $[0,1]$. In such a way, at each step  three new function evaluations of $f$ are performed, in the same way as for the corresponding Runge-Kutta method. On the other hand, $d$ is taken again with integer centered nodes, so that, as with the previous choices of nodes, just one new function evaluation of $g$ is required per step in comparison with the three required for the Runge-Kutta method. More precisely,
\begin{eqnarray}
\left\{\begin{array}{rcl}
\mathbf{c^n}&=&[\frac{1}{2}[-\frac{\sqrt{3}}{2}-1],-\frac{1}{2},\frac{1}{2}[\frac{\sqrt{3}}{2}-1],\frac{1}{2}[-\frac{\sqrt{3}}{2}+1] ,\frac{1}{2} ,\frac{1}{2}[\frac{\sqrt{3}}{2}+1]], \\
 \mathbf{d^n}&=&[-3, -2, -1, 0, 1,2,3], \hspace{5cm} \mbox{Chebyshev-centered method}, n \ge 3.
\end{array} \right.
\label{cheby-center}
\end{eqnarray}

It is clear that no order reduction is observed, i.e. the errors diminish like $O(\tau^6)$ and even more for the choices (\ref{centered}) and (\ref{cheby-center}). The reason for the latter maybe that the error coming from (\ref{lema43}) is so small that what can be seen is the error just associated to the rational method when integrating a linear homogeneous problem. It can be checked that the error constant for Gauss method of order $6$ is so small for these problems that the method nearly behaves as a $7th$-order method.

Moreover, the size of the errors is much smaller than with the Runge--Kutta method with any of the rational choices. As for the comparison in efficiency, Figure \ref{f2} shows that, for the biggest time stepsize, the Runge-Kutta implementation is also more expensive than all rational suggestions. That difference is made smaller when $\tau$ diminishes (perhaps because the linear systems take less time to be solved). In any case, as the errors are so small with the rational methods, with a fixed CPU time, the accuracy which is achieved with the explicit rational method is three orders of magnitude smaller than with the Runge--Kutta method, four orders of magnitude smaller with the implicit rational one and more than five orders of magnitude smaller with the choices (\ref{centered}) and (\ref{cheby-center}). Besides, we remark again that $\dot{g}$ has not been necessary with the natural suggestions.

{\bf Acknowledgements} C.Arranz-Sim\'on and B. Cano have been supported by Ministerio de Ciencia e Innovaci\'on through Project PID2023-147073NB-I00.

\end{document}